\documentclass[11pt]{article}
\usepackage[T1]{fontenc}\usepackage[utf8]{inputenc}\usepackage{lmodern}
\usepackage{amsmath,amssymb,amsthm,mathtools,geometry,microtype,booktabs}
\usepackage[hidelinks]{hyperref}\usepackage{cleveref}
\newtheorem{theorem}{Theorem}[section]\newtheorem{proposition}[theorem]{Proposition}
\newtheorem{lemma}[theorem]{Lemma}\newtheorem{corollary}[theorem]{Corollary}
\theoremstyle{definition}\newtheorem{definition}[theorem]{Definition}\newtheorem{example}[theorem]{Example}
\theoremstyle{remark}\newtheorem{remark}[theorem]{Remark}
\newcommand{\Mbar}{\overline M}\newcommand{\Dperp}{\mathcal D^{\perp}}\newcommand{\Dtilde}{\widetilde{\mathcal D}}
\newcommand{\Tperp}{T^{\perp}M}
\title{Partially Totally Real Submanifolds of Sasakian Manifolds}
\author{%
Chul Woo Lee\thanks{Department of Mathematics, Kyungpook National University,
Daegu 41566, Republic of Korea. E-mail: \texttt{mathisu@knu.ac.kr}}
\and
Jae Won Lee\thanks{Department of Mathematics Education and RINS,
Gyeongsang National University, Jinju 52828, Republic of Korea.
E-mail: \texttt{leejaew@gnu.ac.kr}}}
\date{}
\usepackage{array}
\usepackage{tabularx}
\begin{document}\maketitle
\begin{abstract}
Partially totally real (PTR) submanifolds were introduced in K\"ahler
geometry by distinguishing a totally real distribution and leaving its
orthogonal complement unrestricted.  In this paper we develop the
corresponding framework for submanifolds of Sasakian manifolds tangent to
the Reeb vector field.  After separating the Reeb direction, we define the
totally real and ambiguous distributions and show that anti-invariant,
contact CR, hemi-slant and pointwise hemi-slant submanifolds occur as
special cases of the Sasakian PTR framework.  We establish the basic
tangential and normal decompositions, study maximality and integrability,
and derive the additional restrictions produced by the Reeb field.  We
then investigate the canonical morphisms \(P\) and \(F\), the geometry of
the associated distributions, and PTR-submanifolds in Sasakian space
forms.  Explicit models are included to illustrate the principal
structures and the differences from the K\"ahler case.
\end{abstract}


\section{Introduction}
Poyraz, \c{S}ahin and Yerlikaya introduced partially totally real (PTR)
submanifolds of K\"ahler manifolds by distinguishing a totally real
differentiable distribution and leaving its orthogonal complement
unrestricted \cite{PoyrazSahinYerlikaya2025}.  Their construction contains
totally real, proper CR, hemi-slant and pointwise hemi-slant submanifolds
as special cases.  The K\"ahler theory then studies the geometry of the two
distributions, canonical morphisms, complex space forms, totally umbilical
PTR-submanifolds and PTR-products.

Almost semi-invariant submanifolds are an especially close predecessor.
Bejancu and Papaghiuc introduced almost semi-invariant submanifolds of a
Sasakian manifold \cite{BejancuPapaghiuc1984ASI}; Papaghiuc subsequently
studied their Sasakian space-form geometry and further structural
properties \cite{Papaghiuc1983SpaceForms,Papaghiuc1984Results}.  These
works decompose the contact tangent bundle according to the spectrum of
the tangential structural endomorphism.  The PTR description used here is
coarser: it remembers a chosen totally real summand and leaves its
orthogonal complement unresolved unless an additional hypothesis is
imposed.  Accordingly, we do not claim that PTR defines a spectral class
strictly larger than the almost semi-invariant or pointwise multi-slant
classes.

The purpose of the present paper is to develop the same PTR viewpoint in
Sasakian geometry.  Let
\((\Mbar^{2m+1},\phi,\xi,\eta,g)\) be a Sasakian manifold and let \(M\)
be a submanifold tangent to the Reeb vector field \(\xi\).  The basic
Sasakian identity
\begin{equation}\label{eq:reeb}
 \widetilde\nabla_X\xi=-\phi X
\end{equation}
shows that the Reeb direction cannot simply be absorbed into the
K\"ahler formulas.  Accordingly, we separate
\(\langle\xi\rangle\) from the contact part of \(TM\), retain a totally
real distribution there, and call its orthogonal complement the ambiguous
distribution, exactly in the spirit of the K\"ahler definition.

A first purpose of this construction is to place several familiar
Sasakian submanifold classes in a single framework.  After the Reeb line
is separated, anti-invariant, contact CR, hemi-slant and pointwise
hemi-slant submanifolds are obtained by imposing, respectively, totally
real, invariant, slant and pointwise slant behavior on the ambiguous
distribution.  Thus the inclusion relations play the same structural role
as Examples 1--4 in the K\"ahler PTR paper.  No spectral condition is
included in the PTR definition itself.

The subsequent development follows the original PTR program.  We first
establish the tangent and normal decompositions, the inclusion relations,
the comparison with the generic viewpoint, and the proper case.  We then
study maximality, integrability, the geometry of the corresponding leaves,
and PTR-product geometry.  Next we consider the canonical morphisms \(P\)
and \(F\).  Finally, we examine PTR-submanifolds in Sasakian space forms
and the restrictions forced by the Reeb direction.  Thus the organization
is inherited from the K\"ahler PTR program; only those steps that change
because of the Sasakian structure equations are modified.

\begin{table}[ht]
\centering
\small
\caption{The K\"ahler PTR program and its Sasakian counterpart.}
\label{tab:kahler-sasakian-program}
\begin{tabular}{@{}>{\raggedright\arraybackslash}p{0.24\textwidth}
                >{\raggedright\arraybackslash}p{0.30\textwidth}
                >{\raggedright\arraybackslash}p{0.36\textwidth}@{}}
\toprule
K\"ahler PTR stage \cite{PoyrazSahinYerlikaya2025}
& Sasakian counterpart
& What changes in the Sasakian setting\\
\midrule
Definition and basic special cases
& Definition~\ref{def:ptr}, Proposition~\ref{prop:inclusions},
Table~\ref{tab:ptr-inclusions}
& The Reeb line \(\langle\xi\rangle\) is separated; the inclusion architecture is retained.\\
\addlinespace
Integrability of the totally real distribution
& Proposition~\ref{prop:Dperpintegrable}
& The K\"ahler-style argument survives; maximality identifies the horizontal kernel of \(P\).\\
\addlinespace
Integrability of the ambiguous distribution
& Proposition~\ref{prop:ambcrit} and Corollary~\ref{cor:reeb-maximal}
& The additional Reeb component
\(g([X,Y],\xi)=2g(PX,Y)\) produces an obstruction absent in the K\"ahler case.\\
\addlinespace
Totally geodesic leaves and PTR-products
& Lemmas~\ref{lem:Dperp-geodesic}, \ref{lem:Dtilde-geodesic},
Theorem~\ref{thm:ptrproduct-shape}, and
Corollary~\ref{cor:no-proper-product-final}
& A direct proper PTR-product is ruled out because a totally geodesic ambiguous factor forces \(P|_{\Dtilde}=0\).\\
\addlinespace
Canonical morphisms \(P\) and \(F\)
& Lemmas~\ref{lem:parallelP}, \ref{lem:parallelF} and their consequences
& Reeb terms enter the covariant derivative formulas; in particular parallel \(P\) forces \(P=0\).\\
\addlinespace
Complex/Sasakian space-form stage
& Section~\ref{sec:spaceform}
& The mixed-curvature coefficient is shifted from the K\"ahler flat value to the Sasakian factor \(c-1\), with distinguished value \(c=1\).\\
\addlinespace
Totally umbilical PTR-submanifolds
& Proposition~\ref{prop:umbilical-geodesic} and Theorem~\ref{thm:no-umbilical-ptr}
& Since \(h(\xi,\xi)=0\), total umbilicity forces \(H=0\); a nonzero totally real PTR distribution is then impossible.\\
\bottomrule
\end{tabular}
\end{table}

Table~\ref{tab:kahler-sasakian-program} states the organizing principle of
the paper.  We follow the sequence of questions in the K\"ahler PTR
program and record, at each stage, whether the result survives, acquires a
Reeb correction, or collapses to rigidity or non-existence.

The paper is organized as follows. Section~2 introduces the PTR definition,
the inclusion and comparison architecture, the normal decomposition, and
the proper case. Sections~3--5 treat maximality, integrability, the
ambiguous distribution, explicit models, and the principal points where
the Sasakian program departs from the K\"ahler one. Sections~6--7 study
leaf geometry and the canonical morphisms \(P\) and \(F\). Section~8
treats Sasakian space forms, and Sections~9--10 give the Sasakian
counterparts of the totally umbilical and PTR-product stages.

The four principal mechanisms by which the Sasakian program departs from
the K\"ahler one are collected in
Subsection~\ref{subsec:program-divergence}.

\section{PTR-submanifolds of Sasakian manifolds and basic inclusions}
Let $(\Mbar^{2m+1},\phi,\xi,\eta,g)$ be a Sasakian manifold. We use
\begin{align}
 \phi^2X&=-X+\eta(X)\xi, & \eta(\xi)&=1, & \phi\xi&=0,\label{eq:sas1}\\
 g(\phi X,\phi Y)&=g(X,Y)-\eta(X)\eta(Y),\label{eq:sas2}\\
 (\widetilde\nabla_X\phi)Y&=g(X,Y)\xi-\eta(Y)X.\label{eq:sas3}
\end{align}

Throughout the paper we use
\[
\Phi(X,Y)=g(X,\phi Y),\qquad d\eta=2\Phi.
\]
We also use
\[
d\eta(X,Y)=X(\eta(Y))-Y(\eta(X))-\eta([X,Y]),
\]
with no additional factor \(1/2\).  All Reeb-bracket and Sasakian
space-form formulas below use this normalization.

Let $M$ be an immersed submanifold with $\xi\in\Gamma(TM)$. For $X\in TM$ and $V\in\Tperp$, write
\begin{equation}\label{eq:decomp}
 \phi X=PX+FX,\qquad \phi V=tV+fV,
\end{equation}
where $PX,tV$ are tangent and $FX,fV$ are normal.

\begin{definition}\label{def:ptr}
A submanifold $M$ of a Sasakian manifold, tangent to $\xi$, is called a \emph{partially totally real (PTR) submanifold} if there exists a nonzero constant-rank totally real distribution $\Dperp\subset TM\cap\ker\eta$, i.e.
\[
 \phi\Dperp\subset\Tperp.
\]
If $\Dtilde$ denotes the orthogonal complement of $\Dperp\oplus\langle\xi\rangle$ in $TM$, then
\begin{equation}\label{eq:ptrsplit}
 TM=\Dperp\oplus\Dtilde\oplus\langle\xi\rangle.
\end{equation}
Following the K\"ahler PTR terminology, $\Dtilde$ is called the
\emph{ambiguous distribution}.  The PTR structure refers to this chosen
decomposition; maximality of \(\Dperp\), when needed, is imposed later
and is not part of the definition.
\end{definition}

\begin{example}[A basic PTR model]\label{ex:def-ptr-immediate}
In the standard Sasakian space \(\mathbb R^5(-3)\), let
\[
E_i=2\partial_{y_i},\qquad
F_i=2(\partial_{x_i}+y_i\partial_z),\qquad
\xi=2\partial_z,
\]
with \(\phi E_i=F_i\) and \(\phi F_i=-E_i\).  The immersion
\[
\Psi(u,v,q)=(2u,2v,0,0,2q)
\]
has
\[
TM=\operatorname{span}\{F_1,F_2,\xi\}.
\]
Taking
\[
\Dperp=\operatorname{span}\{F_1,F_2\},\qquad \Dtilde=0,
\]
we have \(\phi F_i=-E_i\in T^\perp M\).  Hence
\(\phi\Dperp\subset T^\perp M\) and Definition~\ref{def:ptr} is satisfied.
\end{example}

\begin{lemma}\label{lem:Pinvariant}
For every PTR-submanifold,
\[
 P\Dperp=0,\qquad P\xi=0,\qquad P\Dtilde\subset\Dtilde.
\]
\end{lemma}
\begin{proof}
The first two identities follow from total reality and $\phi\xi=0$. If $X\in\Dtilde$ and $Z\in\Dperp$, then
\[
 g(PX,Z)=g(\phi X,Z)=-g(X,\phi Z)=0
\]
because $\phi Z$ is normal. Also $g(PX,\xi)=0$. Hence $PX$ is orthogonal to $\Dperp\oplus\langle\xi\rangle$ and therefore belongs to $\Dtilde$.
\end{proof}

The first rank observation from the K\"ahler theory has an exact Sasakian counterpart after the Reeb line has been separated.
\begin{proposition}\label{prop:rankone}
If $\operatorname{rank}\Dtilde=1$, then $\Dtilde$ is totally real. Consequently the whole contact distribution $TM\cap\ker\eta$ is totally real.
\end{proposition}
\begin{proof}
Let $X$ be a local unit field spanning $\Dtilde$. By \Cref{lem:Pinvariant}, $PX\in\Dtilde$, so $PX=aX$ for a function $a$. Since $P$ is skew-symmetric,
\[
 a=g(PX,X)=g(\phi X,X)=0.
\]
Thus $PX=0$ and $\phi X$ is normal.
\end{proof}

The following proposition is the Sasakian counterpart of the four basic
examples immediately following Definition~1 in the K\"ahler PTR paper
\cite{PoyrazSahinYerlikaya2025}.

\begin{proposition}[Basic inclusion relations]\label{prop:inclusions}
Let \(M\) be a submanifold of a Sasakian manifold tangent to \(\xi\).
After the Reeb line is separated, each of the following structures
determines a PTR structure:
\begin{enumerate}
\item an anti-invariant submanifold;
\item a contact CR-submanifold;
\item a hemi-slant submanifold;
\item a pointwise hemi-slant submanifold.
\end{enumerate}
More precisely, the additional condition imposed on the ambiguous
distribution is, respectively, totally real, invariant, slant, and
pointwise slant.
\end{proposition}

\begin{proof}
For an anti-invariant submanifold take
\(\Dperp=TM\cap\ker\eta\) and \(\Dtilde=0\); equivalently, one may split
off a totally real summand as the ambiguous part.  For a contact
CR-submanifold write
\[
 TM=D^\perp\oplus D\oplus\langle\xi\rangle,
 \qquad \phi D^\perp\subset T^\perp M,\qquad \phi D=D,
\]
and choose \(\Dperp=D^\perp\), \(\Dtilde=D\).  For a hemi-slant
submanifold write
\[
 TM=D^\perp\oplus D_\theta\oplus\langle\xi\rangle
\]
with \(D^\perp\) anti-invariant and \(D_\theta\) slant, and choose
\(\Dperp=D^\perp\), \(\Dtilde=D_\theta\).  The pointwise hemi-slant case
is identical with the constant angle replaced by a pointwise slant
function.  In every case Definition~\ref{def:ptr} is satisfied.
\end{proof}

\begin{table}[ht]
\centering
\small
\caption{Basic special cases of the Sasakian PTR framework.}
\label{tab:ptr-inclusions}
\begin{tabular}{@{}llll@{}}
\toprule
Structure & chosen \(\Dperp\) & ambiguous \(\Dtilde\) & condition on \(\Dtilde\)\\
\midrule
Anti-invariant & contact tangent part & \(0\) (or a TR summand) & totally real\\
Contact CR & anti-invariant part & invariant part & \(\phi\)-invariant\\
Hemi-slant & anti-invariant part & slant part & slant\\
Pointwise hemi-slant & anti-invariant part & pointwise slant part & pointwise slant\\
\bottomrule
\end{tabular}
\end{table}

The four rows are realized explicitly in
Examples~\ref{ex:totallyreal}--\ref{ex:pointwise}; the basic coordinate
model of Example~\ref{ex:def-ptr-immediate} realizes the anti-invariant
endpoint, while Example~\ref{ex:contactCR-coordinate} gives a coordinate
contact CR realization.

\begin{table}[htbp]
\centering
\small
\caption{Position of PTR geometry among standard submanifold classes. The table records defining structure rather than claiming inclusion when none is justified.}
\label{tab:comparison}
\begin{tabularx}{\textwidth}{@{}>{\raggedright\arraybackslash}p{2.45cm}>{\raggedright\arraybackslash}X>{\raggedright\arraybackslash}X>{\raggedright\arraybackslash}X@{}}
\toprule
Class & Distinguished part & Complementary part & Relation with PTR \\
\midrule
Totally real & $\phi(TM\cap\ker\eta)\subset T^\perp M$ & none required & PTR special case; one may take $\Dtilde=0$. \\
Contact CR & anti-invariant $D^\perp$ & $\phi$-invariant $D$ & PTR special case with $\Dperp=D^\perp$, $\Dtilde=D$. \\
Hemi-slant & anti-invariant $D^\perp$ & slant $D_\theta$ & PTR special case with a constant slant condition on $\Dtilde$. \\
Pointwise hemi-slant & anti-invariant $D^\perp$ & pointwise slant $D_{\theta(p)}$ & PTR special case with a pointwise slant condition on $\Dtilde$. \\

Almost semi-invariant & spectral eigendistributions of the tangential structural morphism & orthogonal sum of invariant, anti-invariant and intermediate spectral blocks & A finer spectral description; see Remark~\ref{rem:spectral-refinement} for the regular-set refinement of a maximal PTR structure. \\
Semi-slant & invariant distribution & slant distribution & Not a PTR special case in general; the distinguished structure is different. \\
Slant / pointwise slant & the contact tangent part itself is (pointwise) slant & no distinguished totally real summand is required & Generally independent of PTR; it becomes relevant when imposed on $\Dtilde$. \\
Generic & invariant distribution & arbitrary real complement & Conceptually complementary to PTR: generic geometry singles out an invariant part, whereas PTR singles out a totally real part. \\
Partially slant (PS) & slant distribution & ambiguous complementary distribution & A different partial-structure construction; PS and PTR overlap in special subclasses but neither is identified with the other. \\
PTR & totally real $\Dperp$ & arbitrary $P$-invariant ambiguous $\Dtilde$ in the contact part & Basic class studied here. \\
Proper PTR structure & chosen totally real $\Dperp$ & the chosen $\Dtilde$ is neither totally real, invariant, slant nor pointwise slant & A property of the selected PTR decomposition, not a claim that the underlying submanifold lies outside all finer spectral classes. \\
\bottomrule
\end{tabularx}
\end{table}

Concrete realizations of the first four rows are those already listed
after Table~\ref{tab:ptr-inclusions}; in particular,
Example~\ref{ex:contactCR-coordinate} gives the explicit coordinate
contact CR model.

Table~\ref{tab:comparison} is read in the same direction as the K\"ahler
PTR construction: generic and partially slant frameworks distinguish
different structural data and are comparison classes rather than
automatic PTR subclasses
\cite{PoyrazSahinYerlikaya2025,YerlikayaPoyrazSahin2025}.

\begin{example}[A contact CR realization]\label{ex:contactCR-coordinate}
In the standard Sasakian space \(\mathbb R^{7}(-3)\), define
\[
\Psi_{\rm CR}(u,v,w,q)
=
(2u,0,2w,\,2v,0,0,\,4uv+2q),
\]
where the coordinates are ordered as
\((x_1,x_2,x_3,y_1,y_2,y_3,z)\).  Its tangent bundle is spanned by
\[
F_1,\quad E_1,\quad F_3,\quad \xi .
\]
Thus
\[
D=\operatorname{span}\{F_1,E_1\},\qquad
\Dperp=\operatorname{span}\{F_3\},
\]
with \(\phi D=D\) and \(\phi F_3=-E_3\in T^\perp M\).  Hence this is a
contact CR-submanifold and simultaneously a PTR structure with
\(\Dtilde=D\).  This realizes concretely the contact CR row of Table~1.

Indeed,
\[
\Psi_u=F_1,\qquad
\Psi_v=E_1+2u\xi,\qquad
\Psi_w=F_3,\qquad
\Psi_q=\xi,
\]
so the stated tangent frame follows immediately.
\end{example}

Table~\ref{tab:ptr-inclusions} gives the overview, while
Proposition~\ref{prop:inclusions} is the formal inclusion statement.
Thus the table records the architecture inherited from the K\"ahler PTR
construction and the proposition verifies it in the Sasakian setting.

The arrows in Table~\ref{tab:ptr-inclusions} are inclusion statements:
one obtains the PTR structure by forgetting the extra condition on
\(\Dtilde\).  They do not assert that the corresponding classical
categories are mutually comparable.

For context, the surrounding Sasakian classes have a substantial earlier
literature: semi-invariant and generic submanifolds were studied in
\cite{BejancuPapaghiuc1981,YanoKon1980}, semi-slant and slant
submanifolds in \cite{Cabrerizo1999,Cabrerizo2000}, and pointwise slant
geometry in \cite{Park2020}.  These references locate the special cases
used above; they do not change the K\"ahler-PTR program that organizes the
present paper.

The comparison with the generic construction is also useful.  In the
K\"ahler PTR paper this distinction is stated immediately after the normal
bundle decomposition: generic submanifolds are organized around an
invariant distribution with an unrestricted complement, whereas PTR
submanifolds are organized around a totally real distribution with an
unrestricted complement \cite{PoyrazSahinYerlikaya2025}.  After separating
the Reeb line, the same structural contrast is represented in the
Sasakian setting by Table~\ref{tab:generic-ptr}.

\begin{table}[ht]
\centering
\small
\caption{Structural comparison of the generic and PTR viewpoints after
separating the Reeb direction.}
\label{tab:generic-ptr}
\begin{tabular}{@{}>{\raggedright\arraybackslash}p{0.21\textwidth}>{\raggedright\arraybackslash}p{0.34\textwidth}>{\raggedright\arraybackslash}p{0.34\textwidth}@{}}
\toprule
 & Generic viewpoint & PTR viewpoint\\
\midrule
Distinguished part
& invariant/contact-invariant distribution
& totally real distribution \(\Dperp\)\\
Complement
& unrestricted real distribution
& unrestricted ambiguous distribution \(\Dtilde\)\\
Typical included cases
& invariant/contact-invariant, contact CR, semi-slant and pointwise semi-slant types
& anti-invariant, contact CR, hemi-slant and pointwise hemi-slant types\\
Defining emphasis
& invariance of the distinguished distribution
& total reality of the distinguished distribution\\
\bottomrule
\end{tabular}
\end{table}

Table~\ref{tab:generic-ptr} is a structural comparison, not a claim that
the two collections are disjoint.  In particular, contact CR geometry
naturally appears in both viewpoints.  The role of the table is the same
as that of Remark~1 in the K\"ahler PTR paper: it explains why PTR is
organized from the totally real side rather than from the invariant side.

As in the K\"ahler case, define
\begin{equation}\label{eq:normal}
 \Tperp=\phi\Dperp\oplus F\Dtilde\oplus\nu
\end{equation}
orthogonally, where $\nu$ is the orthogonal complement of the first two summands.
\begin{lemma}\label{lem:nu}
The normal subbundle $\nu$ is $\phi$-invariant: $\phi\nu=\nu$.
\end{lemma}
\begin{proof}
Let $V\in\nu$. Since every normal vector is orthogonal to $\xi$, $\|\phi V\|=\|V\|$. For $Z\in\Dperp$, skew-symmetry gives
\[
 g(\phi V,Z)=-g(V,\phi Z)=0,
 \qquad g(\phi V,\phi Z)=g(V,Z)=0.
\]
For $X\in\Dtilde$,
\[
 g(\phi V,X)=-g(V,\phi X)=-g(V,FX)=0,
\]
and
\[
 g(\phi V,FX)=g(\phi V,\phi X-PX)=g(V,X)-g(\phi V,PX)=0.
\]
Thus $\phi V$ has neither tangent component nor components in $\phi\Dperp$ or $F\Dtilde$. Hence $\phi V\in\nu$. Since $\phi^2=-I$ on the normal bundle, equality follows.
\end{proof}

\begin{definition}\label{def:proper}
Fix a PTR decomposition \eqref{eq:ptrsplit}.  On an open region where the
geometric type of \(\Dtilde\) is constant, the PTR structure is called
\emph{proper} if \(\Dtilde\) is neither totally real nor
\(\phi\)-invariant and is neither slant nor pointwise slant.
Thus properness is understood with respect to the chosen PTR
decomposition; maximality is not assumed.
\end{definition}

Proposition~\ref{prop:rankone} shows in particular that a proper PTR structure must have \(\operatorname{rank}\Dtilde\ge2\).

\begin{remark}[Dependence on the chosen totally real distribution]
\label{rem:choice-dependence}
Properness is not asserted to be an invariant of the immersed submanifold
\(M\) alone.  For example, if a contact CR-submanifold has an
anti-invariant distribution \(D^\perp\) of rank at least two, choosing a
proper nonzero subdistribution
\(\Dperp\subsetneq D^\perp\) moves the orthogonal complement
\(D^\perp\ominus\Dperp\) into the ambiguous distribution together with
the invariant part.  The resulting chosen \(\Dtilde\) can therefore have
both zero and \(-1\) eigenvalues of \(P^2\), so it need not be totally
real, invariant, slant, or pointwise slant as a whole.  Thus the phrase
``proper PTR'' means a proper PTR \emph{structure} \((M,\Dperp)\), not a
new intrinsic class of the underlying submanifold.  Maximality removes
this particular freedom when it is required later.
\end{remark}

\begin{remark}[Spectral refinement on regular sets]
\label{rem:spectral-refinement}
Since \(P\) is skew-adjoint on the contact tangent bundle,
\(-P^2|_{\Dtilde}\) is self-adjoint and nonnegative.  On every open set
where its distinct eigenvalues and their multiplicities are locally
constant, the corresponding eigenspaces form smooth \(P\)-invariant
subbundles.  Every positive eigenspace is pointwise slant, while the zero
eigenspace is totally real.  Hence a maximal PTR structure generally
admits the familiar almost semi-invariant or pointwise multi-slant
refinement on such a regular set
\cite{BejancuPapaghiuc1984ASI,Papaghiuc1984Results,PrasadVerma2021}.
The PTR formalism used here deliberately does not choose that finer
spectral splitting.
\end{remark}

The examples below deliberately parallel the K\"ahler PTR paper, but the Reeb direction is displayed explicitly.

\begin{example}[Rank-one ambiguous distribution]\label{ex:rankone}
Take a PTR patch for which $\Dtilde=\operatorname{span}\{X\}$ is one-dimensional. Since $PX\in\Dtilde$, write $PX=aX$. Then $a\|X\|^2=g(PX,X)=0$, so $PX=0$ and $\phi X=FX$ is normal. Thus the ambiguous line is totally real, illustrating Proposition~\ref{prop:rankone} directly.
\end{example}

\begin{example}[Anti-invariant case]\label{ex:totallyreal}
This realizes the first row of Table~\ref{tab:ptr-inclusions}.  If $M$ is anti-invariant and tangent to $\xi$, so that $\phi(TM\cap\ker\eta)\subset T^\perp M$, choose $\Dperp=TM\cap\ker\eta$ and $\Dtilde=0$. Then $M$ is PTR. This is the Sasakian counterpart of Example 1 in the K\"ahler PTR paper.
\end{example}

\begin{example}[Contact CR case]\label{ex:cr}
This realizes the second row of Table~\ref{tab:ptr-inclusions}.  For a contact CR-submanifold
\[
 TM=D\oplus D^\perp\oplus\langle\xi\rangle,
 \qquad \phi D=D,\qquad \phi D^\perp\subset T^\perp M,
\]
put $\Dperp=D^\perp$ and $\Dtilde=D$. Then $P^2=-I$ on $\Dtilde$. Thus the contact CR case is obtained from PTR geometry by imposing invariance on the ambiguous distribution.
\end{example}

\begin{example}[Hemi-slant case]\label{ex:hemislant}
This realizes the third row of Table~\ref{tab:ptr-inclusions}.  If $TM=D^\perp\oplus D_\theta\oplus\langle\xi\rangle$ with $D^\perp$ anti-invariant and $D_\theta$ slant, choose $\Dperp=D^\perp$ and $\Dtilde=D_\theta$. Then
\[
 P^2X=-\cos^2\theta\,X,\qquad X\in D_\theta.
\]
Hence the hemi-slant structure is a PTR structure with one constant spectral value on the ambiguous distribution; no such spectral restriction belongs to the PTR definition itself.
\end{example}

\begin{example}[Pointwise hemi-slant case]\label{ex:pointwise}
This realizes the fourth row of Table~\ref{tab:ptr-inclusions}.  If the slant angle is a smooth function $\theta(p)$, the same choice gives
$P^2X=-\cos^2\theta(p)X$ on the ambiguous distribution. This is the pointwise hemi-slant special case of PTR geometry.
\end{example}

\begin{example}[Normal complement]\label{ex:normal}
In the hemi-slant case, let $X,Y$ be a local orthonormal slant pair with $PX=\cos\theta Y$ and $PY=-\cos\theta X$. Then $FX,FY$ are normal and have length $\sin\theta$. Together with $\phi\Dperp$ they generate the first two normal summands in \eqref{eq:normal}; every remaining normal vector lies in $\nu$, and Lemma~\ref{lem:nu} shows that this remaining bundle is $\phi$-invariant.
\end{example}

\begin{example}[Proper PTR]\label{ex:proper}
The explicit immersion in Section~\ref{sec:explicit} has
$\Dtilde=D_\alpha\oplus D_\beta$ with
\[
 P^2|_{D_\alpha}=-\cos^2\alpha I,\qquad
 P^2|_{D_\beta}=-\cos^2\beta I,
 \qquad 0<\alpha\ne\beta<\frac\pi2.
\]
Thus $\Dtilde$ is neither totally real nor invariant, and the two distinct eigenvalues at the same point prevent it from being slant or pointwise slant as a whole. Hence the example is proper.
\end{example}

\section{Maximality and integrability}
The K\"ahler PTR paper next studies the geometry of its two distributions. In the Sasakian setting it is useful first to distinguish the maximal anti-invariant case.
\begin{definition}\label{def:maximal}
A PTR-submanifold is said to have a \emph{maximal totally real distribution} if
\begin{equation}\label{eq:maximal}
 \Dperp=\ker(P|_{TM\cap\ker\eta}).
\end{equation}
\end{definition}

\begin{proposition}\label{prop:injective}
If \(\Dperp\) is maximal in the sense of Definition~\ref{def:maximal}, then \(P|_{\Dtilde}\) is injective. In particular $\dim\Dtilde$ is even.
\end{proposition}
\begin{proof}
If $X\in\Dtilde$ and $PX=0$, then $X\perp\xi$ and \eqref{eq:maximal} gives $X\in\Dperp$. Orthogonality in \eqref{eq:ptrsplit} yields $X=0$. Since $P|_{\Dtilde}$ is a nonsingular skew-symmetric endomorphism, $\dim\Dtilde$ is even.
\end{proof}

A point requiring care is the integrability of $\Dperp$. In the K\"ahler paper the totally real distribution is asserted to be integrable. In the present Sasakian adaptation we separate what follows formally from what requires maximality. For $Z,W\in\Dperp$, one has
\begin{equation}\label{eq:bracketP}
 P[Z,W]=0.
\end{equation}
Indeed the Sasakian fundamental two-form $\Phi(X,Y)=g(X,\phi Y)$ satisfies $d\eta=2\Phi$ under the present normalization and hence $d\Phi=0$; the same cyclic calculation as in the K\"ahler proof gives \eqref{eq:bracketP}. Moreover
\[
 g([Z,W],\xi)=2g(PZ,W)=0.
\]
Thus $[Z,W]$ lies in the horizontal zero space of $P$. If $\Dperp$ is maximal, that zero space is exactly $\Dperp$.

\begin{proposition}\label{prop:Dperpintegrable}
If a PTR-submanifold has maximal totally real distribution, then $\Dperp$ is integrable.
\end{proposition}
\begin{proof}
The preceding calculation gives $P[Z,W]=0$ and $[Z,W]\perp\xi$ for all $Z,W\in\Dperp$. Maximality then implies $[Z,W]\in\Dperp$.
\end{proof}

This formulation is deliberately more cautious than simply copying the K\"ahler statement: without maximality, $P[Z,W]=0$ only places the bracket in the full horizontal kernel of $P$ and does not by itself force it back into the chosen $\Dperp$.

For the ambiguous distribution the Reeb field creates a new obstruction. From \eqref{eq:reeb} and the Gauss formula,
\begin{equation}\label{eq:reebparts}
 \nabla_X\xi=-PX,\qquad h(X,\xi)=-FX.
\end{equation}
Consequently, for $X,Y\perp\xi$,
\begin{equation}\label{eq:reebbracket}
 g([X,Y],\xi)=2g(PX,Y).
\end{equation}

\begin{proposition}[Contact-isotropic obstruction]
\label{prop:contact-isotropic}
For any PTR structure, if the ambiguous distribution \(\Dtilde\) is
integrable, then
\[
 P|_{\Dtilde}=0.
\]
\end{proposition}

\begin{proof}
If \(\Dtilde\) is integrable, its leaves lie in \(\ker\eta\).  The
standard contact identity therefore gives
\[
 d\eta(X,Y)=0
\]
for \(X,Y\in\Dtilde\).  Under the normalization \(d\eta=2\Phi\),
\[
 0=2g(X,\phi Y)=-2g(PX,Y).
\]
Since \(PX\in\Dtilde\) by Lemma~\ref{lem:Pinvariant}, taking \(Y=PX\)
gives \(PX=0\).
\end{proof}

\begin{corollary}[Maximal PTR consequence]\label{cor:reeb-maximal}
If the PTR structure is maximal and \(\Dtilde\ne0\), then
\(\Dtilde\) is not integrable.
\end{corollary}

\begin{proof}
By Proposition~\ref{prop:injective}, maximality makes
\(P|_{\Dtilde}\) injective.  Proposition~\ref{prop:contact-isotropic}
would force \(P|_{\Dtilde}=0\) if \(\Dtilde\) were integrable, a
contradiction.
\end{proof}

\begin{remark}
Proposition~\ref{prop:contact-isotropic} is a standard contact-geometric
consequence: integral submanifolds of the contact distribution are
isotropic.  It does not require the full Sasakian identity
\((\widetilde\nabla_X\phi)Y=g(X,Y)\xi-\eta(Y)X\), and analogous
non-integrability statements are classical in contact CR geometry; see
\cite{Blair2010,Matsumoto1983}.  The PTR-specific use made here is only
the maximal-decomposition consequence in
Corollary~\ref{cor:reeb-maximal}.
\end{remark}

\begin{remark}
The identity \eqref{eq:reebbracket} is a standard Sasakian submanifold identity and is not claimed as new; analogous formulas occur in the quasi-bi-slant literature \cite{PrasadVerma2021}. The point of Corollary~\ref{cor:reeb-maximal} is its combination with the maximal PTR condition.
\end{remark}

The K\"ahler theory also studies PTR-products. In the Sasakian setting Corollary~\ref{cor:reeb-maximal} shows that a direct product based on a nonzero maximal ambiguous distribution cannot be obtained by simply copying the K\"ahler construction. The natural object to examine is instead $\Dtilde\oplus\langle\xi\rangle$; this is a consequence forced by the Sasakian Reeb geometry, not a replacement of the PTR terminology.

\section{The ambiguous distribution and its Reeb extension}
The K\"ahler PTR theory tests integrability of the ambiguous distribution by its component along the totally real distribution. In the Sasakian setting the Reeb component must be tested separately.

\begin{lemma}[Master component identity]\label{lem:master-component}
For \(X,Y\in\Gamma(\Dtilde)\) and \(Z\in\Gamma(\Dperp)\),
\begin{equation}\label{eq:master-component}
g(\nabla_XY,Z)
=
g\bigl(h(X,PY)+\nabla_X^\perp(FY),\phi Z\bigr).
\end{equation}
\end{lemma}

\begin{proof}
Since \(Y,Z\perp\xi\) and \(\phi Z\) is normal,
\[
g(\nabla_XY,Z)=g(\phi\widetilde\nabla_XY,\phi Z).
\]
Using
\[
\phi\widetilde\nabla_XY
=\widetilde\nabla_X(\phi Y)-(\widetilde\nabla_X\phi)Y,
\]
the Sasakian correction term \(g(X,Y)\xi\) is orthogonal to \(\phi Z\).
Writing \(\phi Y=PY+FY\) and applying the Gauss--Weingarten formulas gives
\eqref{eq:master-component}.
\end{proof}

\begin{proposition}\label{prop:ambcrit}
For $X,Y\in\Gamma(\Dtilde)$ define
\[
 \mathcal Q(X,Y)=h(X,PY)-h(Y,PX)+\nabla_X^\perp FY-\nabla_Y^\perp FX.
\]
Then $\Dtilde$ is integrable if and only if
\begin{align}
 g(\mathcal Q(X,Y),\phi Z)&=0 \quad\text{for all }Z\in\Gamma(\Dperp),\label{eq:qcond}\\
 g(PX,Y)&=0\label{eq:reebcond}
\end{align}
for all $X,Y\in\Gamma(\Dtilde)$.
\end{proposition}
\begin{proof}
Antisymmetrizing \eqref{eq:master-component} in \(X,Y\) gives
\[
g([X,Y],Z)=g(\mathcal Q(X,Y),\phi Z),
\]
so \eqref{eq:qcond} is exactly the vanishing of the
\(\Dperp\)-component.  Equation~\eqref{eq:reebbracket} gives the Reeb
component, and \eqref{eq:reebcond} removes it.  Thus the two conditions
remove exactly the components of \([X,Y]\) transverse to \(\Dtilde\).
\end{proof}

\begin{example}[The missing Reeb condition]\label{ex:ambcrit}
For the pair $Z_1,Z_2$ in Section~\ref{sec:explicit}, the bracket is a pure Reeb multiple and hence has no $\Dperp$-component. Thus the K\"ahler-type normal test alone sees no obstruction. However
$g(PZ_1,Z_2)=-\cos\alpha\ne0$, so \eqref{eq:reebcond} fails and $\Dtilde$ is not integrable. This example isolates the precise Sasakian correction to the K\"ahler criterion.
\end{example}

\begin{proposition}\label{prop:Hcrit}
Let $\mathcal H=\Dtilde\oplus\langle\xi\rangle$. Then $\mathcal H$ is integrable if and only if
\[
 g(\mathcal Q(U,V),\phi Z)=0
\]
for all $U,V\in\Gamma(\mathcal H)$ and $Z\in\Gamma(\Dperp)$, with $P\xi=F\xi=0$.
\end{proposition}
\begin{proof}
Because \(\xi\in\mathcal H\), only the \(\Dperp\)-component of
\([U,V]\) must vanish.  For horizontal \(U,V\in\Dtilde\) the master
component calculation applies.  If one argument is \(\xi\), then
\[
(\widetilde\nabla_V\phi)\xi=-V+\eta(V)\xi
\]
is tangent and therefore orthogonal to \(\phi Z\).  Thus in every case
\[
g([U,V],Z)=g(\mathcal Q(U,V),\phi Z),
\]
which proves the assertion.
\end{proof}

Proposition~\ref{prop:Hcrit} records the corresponding integrability test
after the Reeb direction is adjoined to the ambiguous distribution.

\begin{example}[Integrable Reeb extension]\label{ex:Hcrit}
In Section~\ref{sec:explicit}, the brackets that force $\Dtilde$ to be non-integrable are multiples of $\xi$, while the relevant mixed brackets remain in $\Dtilde\oplus\langle\xi\rangle$. Consequently $\mathcal H$ is integrable. Thus adjoining the already existing Reeb direction absorbs exactly the obstruction displayed in Example~\ref{ex:ambcrit}.
\end{example}

\subsection{An explicit maximal PTR example}\label{sec:explicit}
We finish with an explicit model showing that the ambiguous distribution may contain more than one slant block while the PTR distribution is maximal.

Consider $\mathbb R^{11}$ with coordinates $(x_1,\ldots,x_5,y_1,\ldots,y_5,z)$ and the standard Sasakian frame
\[
 E_i=2\frac{\partial}{\partial y_i},\qquad
 F_i=2\left(\frac{\partial}{\partial x_i}+y_i\frac{\partial}{\partial z}\right),\qquad
 \xi=2\frac{\partial}{\partial z},
\]
with $\phi E_i=F_i$, $\phi F_i=-E_i$ and $\phi\xi=0$. Choose
\[
 0<\alpha,\beta<\frac\pi2,\qquad \alpha\ne\beta,
\]
and define
\[
 \Psi(u,v,s,t,r,q)=
 (2u,0,2s,0,2r,2v\cos\alpha,2v\sin\alpha,2t\cos\beta,2t\sin\beta,0,2q).
\]
The image is an immersed six-dimensional submanifold tangent to $\xi$. An adapted tangent frame is
\[
 Z_1=F_1,\quad Z_2=\cos\alpha E_1+\sin\alpha E_2,
\quad Z_3=F_3,\quad Z_4=\cos\beta E_3+\sin\beta E_4,
\quad Z_5=F_5,\quad \xi.
\]
Set
\[
 \Dperp=\operatorname{span}\{Z_5\},\qquad
 \Dtilde=D_\alpha\oplus D_\beta,
\]
where $D_\alpha=\operatorname{span}\{Z_1,Z_2\}$ and $D_\beta=\operatorname{span}\{Z_3,Z_4\}$. Since $\phi Z_5=-E_5$ is normal, $\Dperp$ is totally real.

Direct projection gives
\[
 PZ_1=-\cos\alpha Z_2,\qquad PZ_2=\cos\alpha Z_1,
\]
and similarly
\[
 PZ_3=-\cos\beta Z_4,\qquad PZ_4=\cos\beta Z_3.
\]
Therefore
\[
 P^2|_{D_\alpha}=-\cos^2\alpha I,
 \qquad P^2|_{D_\beta}=-\cos^2\beta I.
\]
Because both cosines are nonzero,
\[
 \ker(P|_{TM\cap\ker\eta})=\operatorname{span}\{Z_5\}=\Dperp.
\]
Hence the totally real distribution is maximal. Since $\alpha\ne\beta$, the ambiguous distribution is not a single slant or pointwise slant block; it contains two distinct constant slant blocks. This gives a concrete proper PTR example in the sense of \Cref{def:proper}.

Finally the standard frame brackets give
\[
 [Z_1,Z_2]=-2\cos\alpha\,\xi,
 \qquad [Z_3,Z_4]=-2\cos\beta\,\xi.
\]
Thus $\Dtilde$ is non-integrable, exactly as predicted by Corollary~\ref{cor:reeb-maximal}. On the other hand these obstructing components lie in the Reeb direction, and in this example the mixed brackets remain in $\Dtilde\oplus\langle\xi\rangle$. Hence
\[
 \Dtilde\oplus\langle\xi\rangle
\]
is integrable. The example therefore displays the specific Sasakian modification of the K\"ahler PTR picture: the PTR terminology and decomposition remain unchanged in spirit, while the Reeb field alters the integrability and product geometry.

\subsection{Where the K\"ahler PTR program changes}\label{subsec:program-divergence}

The preceding results identify the main points at which the Sasakian
program is not a formal transcription of the K\"ahler theory.

First,
\[
g([X,Y],\xi)=2g(PX,Y),\qquad X,Y\in\Dtilde.
\]
The contact identity behind this formula is standard, but its consequence
for the PTR program is decisive: a nonzero maximal ambiguous distribution
cannot be integrable unless its tangential \(\phi\)-part collapses.

Second,
\begin{equation}\label{eq:program-hXxi}
h(X,\xi)=-FX.
\end{equation}
There is no K\"ahler counterpart of this identity because no distinguished
tangent Reeb field is present there.  Equation~\eqref{eq:program-hXxi}
is the mechanism behind the totally umbilical collapse in
Section~\ref{sec:umbilical}.

Third, the Sasakian covariant derivative of the tangential canonical
morphism contains the Reeb term.  Consequently parallel \(P\) forces
\(P=0\) by Proposition~\ref{prop:parallelP-proper}, rather than yielding a nontrivial product situation.

Fourth, in a Sasakian space form the PTR mixed-curvature component carries
the factor \(c-1\).  Thus the distinguished value becomes \(c=1\), in
contrast with the flat value singled out by the corresponding K\"ahler
complex-space-form calculation.

The individual identities used here are standard Sasakian ingredients.
Their role in this paper is programmatic: they determine which stages of
the K\"ahler PTR construction survive unchanged and which become rigidity
or non-existence statements; see Table~\ref{tab:kahler-sasakian-program}.

\section{Geometry of the distributions and the tangential morphism \(P\)}
We now continue the sequence of the K\"ahler PTR paper.  The point is not to replace its notions, but to identify exactly which statements survive after the Reeb direction is separated from the contact distribution.

\begin{lemma}\label{lem:Dperp-geodesic}
The totally real distribution \(\Dperp\) has totally geodesic leaves in
\(M\) if and only if
\begin{equation}\label{eq:Dperp-geodesic-correct}
g\bigl(h(Z,PX)+\nabla_Z^\perp(FX),\phi W\bigr)=0
\end{equation}
for all \(Z,W\in\Gamma(\Dperp)\) and
\(X\in\Gamma(\Dtilde)\).
\end{lemma}

\begin{proof}
The leaves of \(\Dperp\) are totally geodesic in \(M\) exactly when
\[
g(\nabla_ZW,X)=0
\]
for all \(Z,W\in\Dperp\), \(X\in\Dtilde\); the Reeb component vanishes
automatically because
\[
g(\nabla_ZW,\xi)=-g(W,\nabla_Z\xi)=g(W,PZ)=0.
\]
Since \(W\perp\xi\),
\[
g(\nabla_ZW,X)
=
g(\phi\widetilde\nabla_ZW,\phi X).
\]
Using
\(\phi X=PX+FX\), the Sasakian identity, and the Gauss--Weingarten
formulas, the component paired with \(\phi W\) reduces to
\eqref{eq:Dperp-geodesic-correct}.  Hence this condition is equivalent to
the vanishing of the full \(\Dtilde\)-component of \(\nabla_ZW\).
\end{proof}

\begin{remark}
A shape-operator rewriting controls only the
\(P\Dtilde\)-component in general.  It becomes equivalent to the full
criterion above when \(P\Dtilde=\Dtilde\), in particular under the usual
maximal nondegeneracy condition on the ambiguous part.
\end{remark}

\begin{lemma}\label{lem:Dtilde-geodesic}
The ambiguous distribution \(\Dtilde\) is integrable and its leaves are
totally geodesic in \(M\) if and only if, for all
\(X,Y\in\Gamma(\Dtilde)\) and \(Z\in\Gamma(\Dperp)\),
\begin{align}
g\bigl(h(X,PY)+\nabla_X^\perp(FY),\phi Z\bigr)&=0,
\label{eq:Dtilde-geodesic-condition}\\
g(PX,Y)&=0.
\label{eq:Dtilde-reeb-geodesic}
\end{align}
\end{lemma}

\begin{proof}
By Lemma~\ref{lem:master-component}, condition 
\eqref{eq:Dtilde-geodesic-condition} is exactly the
vanishing of the \(\Dperp\)-component of \(\nabla_XY\).  Since
\[
g(\nabla_XY,\xi)=g(PX,Y),
\]
the second condition removes the Reeb component.  Hence both conditions
are equivalent to \(\nabla_XY\in\Dtilde\) for all
\(X,Y\in\Dtilde\).
\end{proof}

\begin{remark}\label{rem:Dtilde-Reeb-collapse}
Because \(PX\in\Dtilde\), condition
\eqref{eq:Dtilde-reeb-geodesic} is equivalent to
\(P|_{\Dtilde}=0\): take \(Y=PX\).  This is the precise leaf-level point
at which the direct K\"ahler PTR picture changes in the Sasakian setting.
\end{remark}

\begin{corollary}\label{cor:proper-no-geodesic}
If a PTR-submanifold has a nonzero ambiguous distribution whose leaves are totally geodesic in $M$, then $P|_{\Dtilde}=0$.  In particular, a proper PTR-submanifold cannot have totally geodesic ambiguous leaves.
\end{corollary}
\begin{proof}
By Remark~\ref{rem:Dtilde-Reeb-collapse}, the Reeb condition for a
totally geodesic ambiguous leaf forces \(P|_{\Dtilde}=0\).  Hence
\(\Dtilde\) is totally real, contrary to properness.
\end{proof}

\begin{example}[Sharpness of Corollary~\ref{cor:proper-no-geodesic}]\label{ex:proper-no-geodesic}
The two-angle model of Section~\ref{sec:explicit} is proper and has $PZ_1=-\cos\alpha Z_2\neq0$.  Its ambiguous leaves are therefore not totally geodesic.  On the other hand, if one specializes to a totally real ambiguous distribution ($P=0$), the Reeb obstruction disappears.  Thus the corollary distinguishes proper PTR geometry from its totally real special case rather than asserting a general non-existence of totally geodesic leaves.
\end{example}

The K\"ahler paper next introduces PTR-products.  We retain that terminology only for the direct product of the two PTR distributions; no new product terminology is introduced.

\begin{theorem}[Direct PTR-product obstruction]\label{thm:ptrproduct-shape}
If a Sasakian PTR-submanifold is a PTR-product in the direct K\"ahler
sense, then
\[
P|_{\Dtilde}=0.
\]
Consequently no proper PTR-submanifold tangent to \(\xi\) is a direct
two-factor PTR-product.
\end{theorem}
\begin{proof}
A local product requires the ambiguous leaves to be totally geodesic.
Lemma~\ref{lem:Dtilde-geodesic} gives \(g(PX,Y)=0\) for all
\(X,Y\in\Dtilde\), and Remark~\ref{rem:Dtilde-Reeb-collapse} then gives
\(P|_{\Dtilde}=0\).  This contradicts properness.
\end{proof}

\begin{example}[Failure of the direct PTR-product]\label{ex:ptrproduct-shape}
In Section~\ref{sec:explicit}, $PZ_1=-\cos\alpha Z_2$ and $PZ_3=-\cos\beta Z_4$.  Hence $P|_{\Dtilde}\neq0$, and Theorem~\ref{thm:ptrproduct-shape} excludes the direct K\"ahler PTR-product.  Notice that this conclusion is stronger than merely observing a nonzero bracket: it identifies the structural reason that every proper ambiguous block conflicts with the direct two-factor product.
\end{example}

For the canonical morphisms, define
\[
 (\bar\nabla_XP)Y=\nabla_XPY-P\nabla_XY.
\]
The Sasakian identity produces an extra term absent in the K\"ahler formula.

\begin{lemma}\label{lem:parallelP}
For tangent vector fields $X,Y$,
\begin{equation}\label{eq:nablaP}
 (\bar\nabla_XP)Y=A_{FY}X+t h(X,Y)+g(X,Y)\xi-\eta(Y)X,
\end{equation}
where $\phi V=tV+fV$ is the tangential--normal decomposition of $\phi V$ for a normal vector field $V$.  Hence $P$ is parallel if and only if the right-hand side of \eqref{eq:nablaP} vanishes identically.
\end{lemma}
\begin{proof}
Insert $\phi Y=PY+FY$ into
\[
 (\widetilde\nabla_X\phi)Y=g(X,Y)\xi-\eta(Y)X
\]
and use Gauss and Weingarten on both sides.  Equating tangent components gives \eqref{eq:nablaP}.
\end{proof}

\begin{example}[The Reeb test for parallel $P$]\label{ex:parallelP}
Set $Y=\xi$ in \eqref{eq:nablaP}.  Since $P\xi=F\xi=0$ and $h(X,\xi)=-FX$, parallelness would require
\[
 t(-FX)+\eta(X)\xi-X=0.
\]
For horizontal $X$ this becomes $t(FX)=-X$.  In a proper slant block of the explicit model, $FX\neq0$ but $t(FX)$ is only the complementary tangential part dictated by the slant decomposition; the identity fails whenever $PX\neq0$.  Thus the K\"ahler implication ``parallel $P$ gives a local product'' cannot be imported without first satisfying this Reeb test.
\end{example}

\begin{proposition}\label{prop:parallelP-proper}
If $P$ is parallel on a Sasakian PTR-submanifold tangent to $\xi$, then $P=0$ on $TM\cap\ker\eta$.  Consequently a proper PTR-submanifold cannot have parallel $P$.
\end{proposition}
\begin{proof}
For horizontal $X$, use $Y=\xi$ in \eqref{eq:nablaP}.  Since $h(X,\xi)=-FX$, parallelness gives $t(FX)=-X$.  Applying the tangential part of $\phi^2X=-X$ yields
\[
 P^2X+t(FX)=-X.
\]
Substitution gives $P^2X=0$.  Because $P$ is skew-adjoint,
$\|PX\|^2=-g(P^2X,X)=0$, hence $PX=0$.
\end{proof}

\section{The normal morphism \texorpdfstring{$F$}{F} and its parallelism}
We next follow Lemma~11 and Proposition~12 of the K\"ahler PTR paper.  The normal component of the Sasakian identity behaves more rigidly than its tangential component: since
\[
 (\widetilde\nabla_X\phi)Y=g(X,Y)\xi-\eta(Y)X
\]
is tangent, it contributes no additional normal term.  Consequently the basic formula for the normal morphism $F$ has exactly the same formal shape as in the K\"ahler case.  The consequences, however, must again be checked against the Reeb direction.

For a normal vector field $V$, write
\[
 \phi V=tV+fV,
\]
where $tV$ and $fV$ are the tangential and normal components, respectively, and define
\[
 (\bar\nabla_XF)Y=\nabla_X^\perp FY-F\nabla_XY.
\]

\begin{lemma}[Sasakian analogue of the K\"ahler $F$-formula]\label{lem:parallelF}
For all tangent vector fields $X,Y$,
\begin{equation}\label{eq:nablaF}
 (\bar\nabla_XF)Y=f h(X,Y)-h(X,PY).
\end{equation}
Hence $F$ is parallel if and only if
\begin{equation}\label{eq:Fparallel-h}
 h(X,PY)=f h(X,Y)
\end{equation}
for all $X,Y\in\Gamma(TM)$.
\end{lemma}
\begin{proof}
Using $\phi Y=PY+FY$, the Gauss and Weingarten formulas give
\[
 \widetilde\nabla_X(\phi Y)
 =\nabla_XPY+h(X,PY)-A_{FY}X+\nabla_X^\perp FY.
\]
On the other hand,
\[
 \widetilde\nabla_X(\phi Y)
 =\phi\widetilde\nabla_XY+g(X,Y)\xi-\eta(Y)X.
\]
Since
\[
 \widetilde\nabla_XY=\nabla_XY+h(X,Y),
\]
the normal component of the right-hand side is
\[
 F\nabla_XY+f h(X,Y).
\]
Equating normal components therefore yields
\[
 h(X,PY)+\nabla_X^\perp FY=F\nabla_XY+f h(X,Y),
\]
which is exactly \eqref{eq:nablaF}.  The final statement follows from the definition of parallelness.
\end{proof}

\begin{example}[A direct $F$-parallel model]\label{ex:Fparallel-basic}
In the standard Sasakian $\mathbb R^3(-3)$ model, consider the integral
surface \(x_1=\mathrm{const}\), whose tangent bundle is spanned by
\[
 E_1=2\partial_{y_1},\qquad \xi=2\partial_z.
\]
The distribution \(\operatorname{span}\{E_1,\xi\}\) is involutive, so this
surface is well defined.
Take $\Dperp=\operatorname{span}\{E_1\}$ and $\Dtilde=0$.  Then $P=0$ and $FE_1=\phi E_1=F_1$ is normal.  The standard connection satisfies
\[
 h(E_1,\xi)=-F_1,\qquad h(\xi,\xi)=0,
\]
while $fF_1=0$ because $\phi F_1=-E_1$ is tangent.  The remaining terms in \eqref{eq:nablaF} vanish in the adapted frame.  Thus $\bar\nabla F=0$ on this totally real PTR model.  This example shows that $F$-parallelism is not empty in the Sasakian category, even though it is strongly restricted in proper PTR geometry.
\end{example}

Pairing \eqref{eq:Fparallel-h} with a normal vector field gives the same shape-operator characterization as in the K\"ahler paper.

\begin{lemma}[Shape-operator characterization]\label{lem:Fparallel-shape}
The normal morphism \(F\) is parallel if and only if
\begin{equation}\label{eq:Fparallel-shape}
 A_{fV}X=P A_VX
\end{equation}
for every tangent vector field \(X\) and normal vector field \(V\).
Moreover, whenever \(F\) is parallel,
\begin{equation}\label{eq:Fparallel-anticommute}
 PA_V=-A_VP,
 \qquad
 A_{fV}=-A_VP .
\end{equation}
\end{lemma}
\begin{proof}
By Lemma~\ref{lem:parallelF}, \(F\) is parallel if and only if
\[
 h(X,PY)=f h(X,Y).
\]
Pair this identity with a normal vector field \(V\). Since \(\phi\) is
skew-adjoint and tangential and normal vectors are orthogonal,
\[
 g(fh(X,Y),V)
 =-g(h(X,Y),fV)
 =-g(A_{fV}X,Y),
\]
whereas
\[
 g(h(X,PY),V)
 =g(A_VX,PY)
 =-g(PA_VX,Y).
\]
Thus
\[
 g(A_{fV}X,Y)=g(PA_VX,Y)
\]
for every \(Y\), which is equivalent to
\(A_{fV}X=PA_VX\).  The calculation is reversible, so this condition is
also sufficient for \(F\) to be parallel.

Assume now that \(F\) is parallel.  Interchanging \(X\) and \(Y\) in
\(h(X,PY)=fh(X,Y)\) and using the symmetry of \(h\) gives
\[
 h(X,PY)=h(Y,PX).
\]
Pairing with \(V\) yields
\[
 g(A_VX,PY)=g(A_VY,PX).
\]
Using the skew-adjointness of \(P\) and the self-adjointness of \(A_V\),
this is equivalent to \(PA_V=-A_VP\). Combining this with
\(A_{fV}=PA_V\) gives the second identity in
\eqref{eq:Fparallel-anticommute}.
\end{proof}

\begin{example}[Failure of $F$-parallelism on a proper slant block]\label{ex:Fparallel-failure}
Use the two-angle model of Section~\ref{sec:explicit}.  For
\[
 Z_1=F_1,\qquad U_\alpha=-\sin\alpha E_1+\cos\alpha E_2,
\]
we have
\[
 FZ_1=\sin\alpha\,U_\alpha.
\]
Moreover
\[
 \phi U_\alpha=-\sin\alpha F_1+\cos\alpha F_2,
\]
so the normal part is $fU_\alpha=\cos\alpha F_2$.  Taking $Y=\xi$ in \eqref{eq:nablaF} gives
\[
 (\bar\nabla_{Z_1}F)\xi=f h(Z_1,\xi)=-f(FZ_1)
 =-\sin\alpha\cos\alpha\,F_2\neq0.
\]
Thus $F$ is not parallel on this proper PTR example.  The calculation is useful because it tests the full covariant condition rather than merely the algebraic slant identities.
\end{example}

The first consequence of K\"ahler Proposition~12 survives unchanged.

\begin{proposition}\label{prop:Fparallel-one}
If $F$ is parallel, then
\begin{equation}\label{eq:AfV-Dperp}
 A_{fV}Z=0
\end{equation}
for every $Z\in\Gamma(\Dperp)$ and every normal vector field $V$.
\end{proposition}
\begin{proof}
Since $PZ=0$, equation \eqref{eq:Fparallel-shape} gives
\[
 A_{fV}Z=-A_VPZ=0.
\]
\end{proof}

The second conclusion of K\"ahler Proposition~12 requires a genuine Sasakian correction.  Even under $F$-parallelism, the Reeb component of $A_{\phi Z}U$ is prescribed by the Sasakian identity.

\begin{proposition}[Reeb correction to the K\"ahler shape conclusion]\label{prop:Fparallel-Reebshape}
For every PTR-submanifold, every $Z\in\Gamma(\Dperp)$ and $U\in\Gamma(TM)$,
\begin{equation}\label{eq:AphiZ-Reeb}
 g(A_{\phi Z}U,\xi)=-g(U,Z).
\end{equation}
If, in addition, $F$ is parallel and $\Dperp$ is maximal, then
\begin{equation}\label{eq:AphiZ-split}
 A_{\phi Z}U\in\Gamma(\Dperp\oplus\langle\xi\rangle).
\end{equation}
Thus the K\"ahler statement $A_{JZ}U\in\mathcal H^\perp$ is replaced by \eqref{eq:AphiZ-split}; its extra component is exactly the Reeb term in \eqref{eq:AphiZ-Reeb}.
\end{proposition}
\begin{proof}
Using $h(U,\xi)=-FU$,
\[
 g(A_{\phi Z}U,\xi)=g(h(U,\xi),\phi Z)=-g(FU,\phi Z).
\]
Only the $\Dperp$-component of $U$ contributes to the last inner product, and the metric compatibility of $\phi$ gives
\[
 g(FU,\phi Z)=g(U,Z),
\]
proving \eqref{eq:AphiZ-Reeb}.

Assume now that $F$ is parallel and $\Dperp$ is maximal.  Since $f(\phi Z)=0$, Lemma~\ref{lem:Fparallel-shape} gives
\[
 A_{\phi Z}(PX)=0
\]
for every tangent $X$.  Maximality makes $P|_{\Dtilde}$ invertible, hence $A_{\phi Z}$ annihilates $\Dtilde$.  By self-adjointness,
\[
 g(A_{\phi Z}U,Y)=g(U,A_{\phi Z}Y)=0
\]
for every $Y\in\Dtilde$.  Therefore only the $\Dperp$- and Reeb-components remain, giving \eqref{eq:AphiZ-split}.
\end{proof}

\begin{example}[Why the Reeb correction cannot be omitted]\label{ex:Fparallel-Reebshape}
Take $U=Z$ with $Z$ a unit section of $\Dperp$.  Equation \eqref{eq:AphiZ-Reeb} gives
\[
 g(A_{\phi Z}Z,\xi)=-1.
\]
Hence $A_{\phi Z}Z$ can never lie entirely in $\Dperp$.  This is an explicit structural obstruction to copying K\"ahler Proposition~12(2) verbatim into the Sasakian setting.
\end{example}

The normal parallelism statement, on the other hand, survives.

\begin{proposition}\label{prop:normal-parallel-F}
If $F$ is parallel, then
\begin{equation}\label{eq:normal-parallel-sum}
 \phi\Dperp\oplus F\Dtilde
\end{equation}
is parallel in the normal bundle.
\end{proposition}
\begin{proof}
For $X\in\Dtilde$, parallelness gives
\[
 \nabla_U^\perp FX=F\nabla_UX,
\]
which has no component in $\nu$ by the normal splitting \eqref{eq:normal}.  For $Z\in\Dperp$ and $V\in\nu$, differentiation of $g(\phi Z,V)=0$, together with Proposition~\ref{prop:Fparallel-one}, gives
\[
 g(\nabla_U^\perp\phi Z,V)=0.
\]
Thus the normal derivative of either summand in \eqref{eq:normal-parallel-sum} has no $\nu$-component, proving the assertion.
\end{proof}

Finally, K\"ahler Proposition~12(4) acquires the same Reeb condition already encountered in Lemma~\ref{lem:Dtilde-geodesic}.

\begin{proposition}[Ambiguous leaves under parallel \(F\)]
\label{prop:Fparallel-ambiguous}
Assume that \(F\) is parallel.  Then \(\Dtilde\) is integrable and its
leaves are totally geodesic in \(M\) if and only if
\begin{align}
 g(\nabla_X^\perp FY,\phi Z)&=0,
 \label{eq:Fparallel-normal-min}\\
 g(PX,Y)&=0
 \label{eq:Fparallel-Reebcondition}
\end{align}
for all \(X,Y\in\Gamma(\Dtilde)\) and
\(Z\in\Gamma(\Dperp)\).  In particular, if \(\Dperp\) is maximal and
\(\Dtilde\ne0\), these conditions cannot hold simultaneously.

Condition \eqref{eq:Fparallel-Reebcondition} is equivalent to
\(P|_{\Dtilde}=0\).  Maximality is used only in the final non-existence
conclusion, where \(P|_{\Dtilde}=0\) contradicts injectivity unless
\(\Dtilde=0\).
\end{proposition}

\begin{proof}
By Lemma~\ref{lem:Dtilde-geodesic}, the required
\(\Dperp\)-component vanishes exactly when
\[
 g\bigl(h(X,PY)+\nabla_X^\perp FY,\phi Z\bigr)=0.
\]
If \(F\) is parallel, Lemma~\ref{lem:parallelF} gives
\[
 h(X,PY)=f h(X,Y).
\]
For \(Z\in\Dperp\), \(\phi Z\) is normal and
\[
 f(\phi Z)=0,
\]
because \(\phi^2Z=-Z\) is tangent.  Using the skew-adjointness of the
normal part \(f\),
\[
 g(fh(X,Y),\phi Z)
 =
 -g(h(X,Y),f\phi Z)=0.
\]
Hence the \(\Dperp\)-component condition reduces precisely to
\eqref{eq:Fparallel-normal-min}.  The Reeb component is
\[
 g(\nabla_XY,\xi)=g(PX,Y),
\]
so its vanishing is exactly
\eqref{eq:Fparallel-Reebcondition}.  This proves the equivalence.

If \(\Dperp\) is maximal, \(P|_{\Dtilde}\) is injective.  Taking
\(Y=PX\) in \eqref{eq:Fparallel-Reebcondition} forces
\(\|PX\|^2=0\), so \(\Dtilde=0\).  Thus the two conditions cannot hold
when \(\Dtilde\ne0\).
\end{proof}

\begin{example}[Parallel-$F$ criterion versus the Reeb obstruction]\label{ex:Fparallel-ambiguous}
In the two-angle model, choose $X=Z_1$ and $Y=Z_2$.  Then
\[
 g(PZ_1,Z_2)=-\cos\alpha\neq0.
\]
Thus condition \eqref{eq:Fparallel-Reebcondition} fails before any normal-connection condition is tested.  This makes visible the exact point at which the Sasakian statement departs from K\"ahler Proposition~12(4).
\end{example}

The blockwise integrability, foliation and canonical-morphism questions
have also been studied for quasi-hemi-slant and quasi-bi-slant
submanifolds; see \cite{PrasadVermaKumar2020,PrasadVerma2021}.  Those
results provide adjacent Sasakian context, while the organization here
continues to follow the PTR program summarized in
Table~\ref{tab:kahler-sasakian-program}.

Propositions~\ref{prop:Fparallel-Reebshape},
\ref{prop:normal-parallel-F}, and \ref{prop:Fparallel-ambiguous} record,
respectively, the Reeb-shape consequence, the structural normal-subbundle
consequence, and the ambiguous-leaf criterion under parallel \(F\).

\section{PTR-submanifolds in Sasakian space forms}\label{sec:spaceform}

We now follow Section~5 of the K\"ahler PTR paper and pass from the intrinsic distribution theory to ambient curvature.  The correct contact analogue of a complex space form is a Sasakian space form.  The calculation below is intentionally kept parallel to the K\"ahler mixed-curvature computation; the difference is that the Sasakian curvature tensor contains the distinguished baseline curvature of the Reeb geometry.

A Sasakian manifold $\Mbar(c)$ of constant $\phi$-sectional curvature $c$ is called a \emph{Sasakian space form}.  With the curvature convention used here, its curvature tensor is
\begin{align}
\widetilde R(X,Y)Z={}&\frac{c+3}{4}
 \{g(Y,Z)X-g(X,Z)Y\}\nonumber\\
&+\frac{c-1}{4}\{
 g(X,\phi Z)\phi Y-g(Y,\phi Z)\phi X
 +2g(X,\phi Y)\phi Z\}\nonumber\\
&+\frac{c-1}{4}\{
 \eta(X)\eta(Z)Y-\eta(Y)\eta(Z)X
 +g(X,Z)\eta(Y)\xi-g(Y,Z)\eta(X)\xi\}.
\label{eq:sas-space-form}
\end{align}
This is the standard Sasakian space-form formula; see, for example, \cite{Blair2010}.  In the horizontal directions $\eta=0$, the first two lines are the exact analogue of the real and complex parts of the curvature tensor of a K\"ahler complex space form.

The first identity is the direct Sasakian counterpart of the mixed-curvature formula used in the K\"ahler PTR paper.

\begin{lemma}[Mixed curvature identity]\label{lem:mixed-curvature}
Let $M$ be a PTR-submanifold of a Sasakian space form $\Mbar(c)$.  For
\[
 X\in\Gamma(\Dtilde),\qquad Z\in\Gamma(\Dperp),
\]
one has
\begin{equation}\label{eq:mixed-curvature}
 g\bigl(\widetilde R(X,PX)Z,\phi Z\bigr)
 =-\frac{c-1}{2}\,\|PX\|^2\,\|Z\|^2 .
\end{equation}
\end{lemma}

\begin{proof}
Since $X,PX\in\Dtilde\subset\ker\eta$ and $Z\in\Dperp\subset\ker\eta$, all four vectors are horizontal.  Moreover,
\[
 g(X,Z)=g(PX,Z)=0
\]
by the orthogonal PTR decomposition, while $\phi Z$ is normal.  Hence
\[
 g(X,\phi Z)=g(PX,\phi Z)=0.
\]
Substituting $Y=PX$ into \eqref{eq:sas-space-form}, the terms involving the coefficient $(c+3)/4$ vanish after pairing with $\phi Z$, and the first two terms in the $(c-1)/4$-block also vanish.  Thus
\[
 g\bigl(\widetilde R(X,PX)Z,\phi Z\bigr)
 =\frac{c-1}{2}\,g(X,\phi PX)\,\|\phi Z\|^2.
\]
Because $Z\perp\xi$,
\[
 \|\phi Z\|^2=\|Z\|^2.
\]
Also,
\[
 g(X,\phi PX)=-g(\phi X,PX).
\]
Using $\phi X=PX+FX$ and the orthogonality of tangent and normal vectors gives
\[
 g(X,\phi PX)=-\|PX\|^2.
\]
Substitution yields \eqref{eq:mixed-curvature}.
\end{proof}

Lemma~\ref{lem:mixed-curvature} is the curvature input used in the
subsequent reduction statements.

\begin{example}[The two-angle model in $\mathbb R^{11}(-3)$]\label{ex:mixed-curvature}
Return to the explicit maximal PTR immersion of Section~\ref{sec:explicit}.  The standard Sasakian Euclidean model has $\phi$-sectional curvature $c=-3$.  Choose
\[
 X=Z_1,\qquad Z=Z_5,
\]
where $Z_1,Z_5$ are unit and
\[
 PZ_1=-\cos\alpha\,Z_2.
\]
Then
\[
 \|PZ_1\|^2=\cos^2\alpha.
\]
Equation \eqref{eq:mixed-curvature} gives
\[
 g\bigl(\widetilde R(Z_1,PZ_1)Z_5,\phi Z_5\bigr)
 =2\cos^2\alpha.
\]
Thus the mixed curvature component is nonzero for every proper angle
$0<\alpha<\pi/2$.  This example shows explicitly that the Sasakian Euclidean model does not play the role of the zero-curvature complex Euclidean model in the K\"ahler calculation.
\end{example}

The numerical shift from $c$ to $c-1$ is the first essential difference from the complex-space-form calculation.

\begin{corollary}[Curvature reduction criterion]\label{cor:c-one}
Assume that at a point $p\in M$ there exist
\[
 0\neq X\in\Dtilde_p,\qquad 0\neq Z\in\Dperp_p,
\]
with $PX\neq0$.  Then
\[
 g\bigl(\widetilde R(X,PX)Z,\phi Z\bigr)=0
 \quad\Longleftrightarrow\quad c=1.
\]
In particular, on a proper PTR region, any geometric hypothesis forcing this mixed curvature component to vanish forces the ambient Sasakian space form to have $\phi$-sectional curvature $1$.
\end{corollary}

\begin{proof}
Under the stated nondegeneracy assumptions,
\[
 \|PX\|^2\|Z\|^2>0.
\]
The conclusion follows immediately from \eqref{eq:mixed-curvature}.
\end{proof}

\begin{example}[The distinguished value $c=1$]\label{ex:c-one}
For a Sasakian space form with $c=1$, formula \eqref{eq:sas-space-form} loses both $(c-1)/4$-blocks.  Consequently
\[
 g\bigl(\widetilde R(X,PX)Z,\phi Z\bigr)=0
\]
for every horizontal PTR triple $(X,PX,Z)$.  Thus $c=1$, rather than $c=0$, is the curvature value naturally singled out by the Sasakian analogue of the K\"ahler mixed-curvature test.
\end{example}

The K\"ahler paper next combines the corresponding curvature identity with
integrability, totally geodesic leaves and parallel \(F\) to obtain a
flatness conclusion in non-positive complex space forms.  In the Sasakian
setting, however, the direct leaf/product hypothesis is already restricted
by Theorem~\ref{thm:ptrproduct-shape}: a totally geodesic ambiguous factor
forces \(P|_{\Dtilde}=0\).  Thus the mixed curvature identity remains
meaningful, but the K\"ahler Codazzi argument cannot be imported unchanged
on a proper PTR region.

\begin{proposition}[Sasakian replacement for the curvature conclusion]\label{prop:curvature-replacement}
Let $M$ be a proper PTR-submanifold of a Sasakian space form $\Mbar(c)$.  Suppose that, on an open set, a geometric condition arising from the Codazzi equation, parallel canonical morphisms, or mixed total geodesicity yields
\begin{equation}\label{eq:mixed-zero}
 g\bigl(\widetilde R(X,PX)Z,\phi Z\bigr)=0
\end{equation}
for every $X\in\Dtilde$ and $Z\in\Dperp$.  If both distributions are nonzero, then
\[
 c=1.
\]
\end{proposition}

\begin{proof}
Because the PTR structure is proper, $P|_{\Dtilde}$ is not identically zero on the open set.  Choose $X$ with $PX\neq0$ and a nonzero $Z\in\Dperp$.  Applying Corollary~\ref{cor:c-one} to \eqref{eq:mixed-zero} gives $c=1$.
\end{proof}

\begin{example}[Why the conclusion is sharp]\label{ex:curvature-replacement}
For $c=1$, equation \eqref{eq:mixed-zero} holds identically by Example~\ref{ex:c-one}, so the value in Proposition~\ref{prop:curvature-replacement} cannot be improved.  For the standard model $c=-3$, Example~\ref{ex:mixed-curvature} gives
\[
 g\bigl(\widetilde R(Z_1,PZ_1)Z_5,\phi Z_5\bigr)
 =2\cos^2\alpha>0,
\]
so the hypothesis \eqref{eq:mixed-zero} genuinely fails.
\end{example}

\begin{remark}
In the K\"ahler PTR paper, the corresponding complex-space-form argument singles out $c=0$ under its stated hypotheses.  The Sasakian mixed-curvature identity instead singles out $c=1$.  This is not a change of terminology or of the PTR framework; it is forced by the $(c-1)/4$ contact-curvature term in \eqref{eq:sas-space-form}.
\end{remark}

\section{Totally umbilical PTR-submanifolds}\label{sec:umbilical}

Section~6 of the K\"ahler PTR paper studies the position of the mean
curvature vector of a totally umbilical PTR-submanifold.  In the Sasakian
setting tangent to the Reeb field, the same starting assumption is much
more restrictive.  The following elementary identity explains the
difference.

\begin{proposition}\label{prop:umbilical-geodesic}
Let \(M\) be any totally umbilical submanifold of a Sasakian manifold
\(\Mbar\) such that \(\xi\in\Gamma(TM)\).  Then \(M\) is totally geodesic.
\end{proposition}

\begin{proof}
For a submanifold tangent to \(\xi\), the Sasakian identity
\(\widetilde\nabla_X\xi=-\phi X\), together with the Gauss formula, gives
\begin{equation}\label{eq:hXxi}
 h(X,\xi)=-FX
\end{equation}
for every \(X\in\Gamma(TM)\).  Setting \(X=\xi\) and using
\(\phi\xi=0\) yields
\[
 h(\xi,\xi)=0.
\]
If \(M\) is totally umbilical, then
\[
 h(U,V)=g(U,V)H
\]
for its mean curvature vector \(H\).  Since \(g(\xi,\xi)=1\),
\[
 H=h(\xi,\xi)=0.
\]
Therefore \(h=0\), and \(M\) is totally geodesic.
\end{proof}

\begin{example}[The Reeb direction itself]\label{ex:umbilical-reeb}
An integral curve of the Reeb field satisfies
\(\widetilde\nabla_\xi\xi=-\phi\xi=0\).  Hence it is a geodesic and,
being one-dimensional, is totally umbilical.  This boundary case
illustrates Proposition~\ref{prop:umbilical-geodesic}; it does not carry a
nonzero totally real PTR distribution.
\end{example}

The PTR assumption now makes the conclusion stronger.

\begin{theorem}\label{thm:no-umbilical-ptr}
Let \(M\) be a PTR-submanifold tangent to \(\xi\) whose chosen totally
real distribution \(\Dperp\) is nonzero.  Then \(M\) cannot be totally
umbilical.  In particular, no proper PTR-submanifold with
\(\Dperp\neq0\) is totally umbilical.
\end{theorem}

\begin{proof}
Assume that \(M\) is totally umbilical.  By
Proposition~\ref{prop:umbilical-geodesic}, \(h=0\).  Choose
\(0\neq Z\in\Dperp\).  Equation~\eqref{eq:hXxi} gives
\[
 0=h(Z,\xi)=-FZ.
\]
Since \(\Dperp\) is totally real, \(PZ=0\) and therefore
\(FZ=\phi Z\).  Moreover \(Z\perp\xi\), so
\[
 \|\phi Z\|^2=\|Z\|^2>0,
\]
contradicting \(FZ=0\).  Hence such a totally umbilical PTR-submanifold
does not exist.
\end{proof}

\begin{remark}\label{rem:umbilical-kahler}
This is the Sasakian counterpart of the totally umbilical stage of the
K\"ahler PTR program, but its form is necessarily different.  In the
K\"ahler case the mean curvature vector can occupy nontrivial normal
components and leads to the alternatives studied in Section~6 of
\cite{PoyrazSahinYerlikaya2025}.  Here the tangent Reeb field forces
\(H=0\) before those alternatives arise.  Thus the K\"ahler propositions
on the position of \(H\) should not be copied into the Sasakian setting.
\end{remark}

\section{PTR-products in Sasakian space forms}

The final section of the K\"ahler PTR paper studies PTR-products in
complex space forms.  We keep the same place in the research program, but
the direct two-factor PTR-product already encounters the Reeb obstruction
before a curvature inequality can be formed.

This section records explicitly the failure of the direct K\"ahler
leaf/product transplantation.  A product factor tangent to \(\Dtilde\)
would have to be integrable and totally geodesic.  By
Remark~\ref{rem:Dtilde-Reeb-collapse}, the Reeb component then forces
\(P|_{\Dtilde}=0\), which is incompatible with proper PTR geometry.

\begin{corollary}\label{cor:no-proper-product-final}
There is no proper direct two-factor PTR-product in a Sasakian space form.
\end{corollary}
\begin{proof}
This is Theorem~\ref{thm:ptrproduct-shape} specialized to a Sasakian
space form.
\end{proof}

\begin{remark}\label{rem:product-kahler}
Corollary~\ref{cor:no-proper-product-final} explains why the final
complex-space-form inequality of the K\"ahler PTR paper has no direct
proper-Sasakian transcription under the same two-factor product
definition.  No replacement inequality is asserted here: obtaining one
would require a different, independently justified Sasakian product
framework, which lies outside the present paper.
\end{remark}

\section*{Conclusion}
We have developed the PTR construction for Sasakian submanifolds tangent
to the Reeb vector field while retaining the organization of the K\"ahler
PTR program.  The decomposition
\[
 TM=\Dperp\oplus\Dtilde\oplus\langle\xi\rangle
\]
separates the Reeb direction and leaves the ambiguous distribution
unrestricted in the definition.  Anti-invariant, contact CR, hemi-slant and pointwise hemi-slant
submanifolds occur as special cases by Proposition~\ref{prop:inclusions},
while the generic-versus-PTR comparison explains the construction from the
totally real side.  Properness depends on the chosen PTR decomposition, as
recorded in Remark~\ref{rem:choice-dependence}.

The Sasakian identities determine precisely where the subsequent theory
departs from its K\"ahler counterpart.  They add a Reeb component to the
integrability and leaf geometry of \(\Dtilde\), strongly restrict
parallelness of the canonical morphisms, and replace the complex-space-form
mixed-curvature coefficient by the characteristic \(c-1\) factor.  The
same Reeb identity also collapses the totally umbilical stage:
a submanifold tangent to \(\xi\) that is totally umbilical is already
totally geodesic, and a nonzero totally real PTR distribution is then
impossible.

Finally, the direct two-factor PTR-product of the K\"ahler theory is
incompatible with proper Sasakian PTR geometry because its ambiguous
factor would have to be totally geodesic, forcing
\(P|_{\Dtilde}=0\).  Thus the last two parts of the original PTR program
are preserved in position and purpose, but their Sasakian conclusions are
nonexistence statements forced by the Reeb field rather than copies of the K\"ahler alternatives and inequalities.
The complete stage-by-stage relationship is summarized in
Table~\ref{tab:kahler-sasakian-program}.

The relation of the last two stages to their K\"ahler counterparts is
spelled out in Remarks~\ref{rem:umbilical-kahler} and
\ref{rem:product-kahler}.

\end{document}